\documentclass[fleqn, pdftex]{article} 

\usepackage{amsmath,amsfonts,latexsym,amssymb,amsthm}
\usepackage{tikz}

\title{Visser frames for sublogics of $\IL$}
\author{Yuya Okawa\footnote{Email: math.y.okawa@gmail.com}
\footnote{Graduate School of Science and Engineering, Chiba University, Japan} and Taishi Kurahashi\footnote{Email: kurahashi@people.kobe-u.ac.jp}
\footnote{Graduate School of System Informatics,
Kobe University, Japan}}
\date{}

\theoremstyle{plain}
\newtheorem{thm}{Theorem}[section]
\newtheorem{lem}[thm]{Lemma}
\newtheorem{prop}[thm]{Proposition}
\newtheorem{cor}[thm]{Corollary}
\newtheorem{fact}[thm]{Fact}
\newtheorem{prob}[thm]{Problem}
\newtheorem*{cl}{Claim}

\theoremstyle{definition}
\newtheorem{defn}[thm]{Definition}

\newtheorem{rem}[thm]{Remark}

\newcommand{\PA}{\mathbf{PA}}

\newcommand{\CL}{\mathbf{CL}}
\newcommand{\IL}{\mathbf{IL}}
\newcommand{\ILM}{\mathbf{ILM}}

\newcommand{\Sub}{\mathrm{Sub}}

\newcommand{\seq}[1]{\langle#1\rangle}

\newcommand{\G}[1]{\mathbf{L#1}}
\newcommand{\J}[1]{\mathbf{J#1}}
\newcommand{\R}[1]{\mathbf{R#1}}
\newcommand{\pre}{\mathsf{p}}

\begin{document}

\maketitle

\begin{abstract}
We study the modal completeness and the finite frame property of several sublogics of the logic $\mathbf{IL}$ of interpretability with respect to Visser frames, which are also called simplified Veltman frames. 
Among other things, we prove that the logic $\mathbf{CL}$ of conservativity has the finite frame property with respect to that frames. 
This is an affirmative solution to Ignatiev's problem. 
\end{abstract}

\section{Introduction}

Throughout the introduction, let $T$ stand for consistent computably enumerable theories extending Peano Arithmetic $\PA$. 
The notion of interpretability has been investigated through the framework of modal logic. 
The language of interpretability logics is obtained from that of modal propositional logic by adding the binary modal operator $\rhd$, where the modal formula $A \rhd B$ is intended to mean ``$T+B$ is interpretable in $T+A$". 
The logic $\IL$ of interpretability is a base logic to investigate several interpretability logics. 
Berarducci~\cite{Ber90} and Shavrukov~\cite{Sha88} independently proved that the extension $\ILM$ of $\IL$ is arithmetically complete with respect to arithmetical interpretations based on any $\Sigma_{1}$-sound theory $T$.  
It is known that the notions of interpretability and $\Pi_1$-conservativity coincide for computably enumerable extensions of $\PA$, so the logic $\ILM$ is also the logic of $\Pi_{1}$-conservativity (cf.~H\'ajek and Montagna~\cite{HajMon90}), where the formula $A \rhd B$ is read as ``$T+B$ is $\Pi_1$-conservative over $T+A$".  
In this context, Ignatiev~\cite{Ign91} further investigated the logic of $\Gamma$-conservativity for $\Gamma \in \{\Sigma_{n}, \Pi_{n} : n \geq 1\}$ and introduced the logic $\CL$ of conservativity, that is a sublogic of $\IL$, as a base logic to investigate several conservativity logics.

The logics $\IL$ and $\CL$ are complete with respect to a relational semantics, which is called Veltman semantics. 
We say that a triple $(W, R, \{S_{x}\}_{x \in W})$ is an \textit{$\IL$-frame} if $W$ is a non-empty set, $R$ is a transitive and conversely well-founded binary relation on $W$, and for each $x \in W$, $S_{x}$ is a transitive and reflexive binary relation on $R[x] :=\{y \in W \mid xRy\}$ satisfying the condition ($\dagger$): $(\forall x, y, z \in W)(xRy \, \& \, yRz \Rightarrow y S_{x} z)$. 
Also, a \textit{$\CL$-frame} is obtained by removing the condition $(\dagger)$ from the definition of an $\IL$-frame. 
Then, de Jongh and Veltman~\cite{deJVel90} proved that $\IL$ is sound and complete with respect to the class of all finite $\IL$-frames. 
Moreover, Ignatiev~\cite{Ign91} proved that $\CL$ is sound and complete with respect to the class of  all finite $\CL$-frames.

Visser~\cite{Vis88} introduced and investigated the notion of Visser frames, which are also called simplified Veltman frames. 
An \textit{$\IL$-Visser frame} is a triple $(W, R, S)$ where $(W, R)$ satisfies the same condition as in the definition of Veltman frames and $S$ is a transitive and reflexive binary relation on $W$ satisfying the condition ($\dagger$'): $R \subseteq S$. 
A \textit{$\CL$-Visser frame} is obtained by removing the condition ($\dagger$') from the definition of an $\IL$-Visser frame. 
Then, Visser~\cite{Vis88} proved that $\IL$ is sound and complete with respect to the class of all $\IL$-Visser frames. 
On the other hand, unlike the case of Veltman frames, Visser~\cite{Vis97} proved that $\IL$ does not have finite frame property with respect to $\IL$-Visser frames. 
Ignatiev~\cite{Ign91} proved that the modal completeness theorem of $\CL$, that is, $\CL$ is sound and complete with respect to the class of all $\CL$-Visser frames\footnote{Ignatiev actually proved a stronger result that $\CL$ is sound and complete with respect to the class of all $\CL$-Visser frames where $S$ is symmetric.}. 
Then, he proposed the following problem. 
\begin{prob}[Ignatiev~{\cite[p.~33]{Ign91}}]\label{P1}
Does the logic $\CL$ have finite frame property with respect to $\CL$-Visser frames?
\end{prob}

By using the modal completeness of $\CL$ with respect to Visser frames, Ignatiev also proved the arithmetical completeness of $\CL$. 

\begin{thm}[Ignatiev~{\cite[Theorem 2]{Ign91}}]\label{ACT}
$\CL$ is exactly the intersection of logics of $\Gamma$-conservativity for all suitable classes $\Gamma$ of sentences. 
\end{thm}

However, the authors at least think that there is one fault in his proof of Theorem \ref{ACT}.  
His proof is proceeded by embedding a counter $\CL$-Visser model $(W, R, S, \Vdash)$ obtained by his modal completeness theorem into arithmetic as in the proof of Solovay~\cite{Sol76}, where $W$ is an infinite set. 
Then, Ignatiev stated that for each $x \in W$, finite set $X \subseteq W$, and modal formula $C$, the following relation is represented by a $\Delta_{0}$ formula:
\[
(\exists {z} \in W)\bigl(xRz \land (\exists {y} \in X)({y} S {z}) \land {z} \Vdash C\bigr).
\]
However, since $W$ is an infinite set, the formula ``$\exists {z} \in W$" is not a $\Delta_{0}$ formula in general. 
In addition, the author at least think that this formula is not represented by any $\Delta_{0}$ formula as long as we use his model. 
On the other hand, if $\CL$ has the finite frame property with respect to Visser frames, then the above formula is obviously represented by a $\Delta_{0}$ formula because $W$ is a finite set. 
Therefore, if Problem~\ref{P1} is affirmatively solved, then Ignatiev's proof of Theorem \ref{ACT} is completely correct.

Veltman semantics was also introduced for other sublogics of $\IL$ than $\CL$. 
Visser~\cite{Vis88} introduced  \textit{$\IL^- \!$-frames} (or \textit{Veltman prestructures}) which are general notions of $\IL$-frames and $\CL$-frames. 
Kurahashi and Okawa~\cite{KO20} introduced the sublogic $\IL^-$ of $\IL$, and proved that $\IL^-$ is sound and complete with respect to the class of  all finite $\IL^- \!$-frames. 
In~\cite{KO20}, several sublogics of $\IL$ were introduced by adding several axioms to $\IL^-$, and their completeness with respect to $\IL^- \!$-frames was investigated. 

In this context, we  investigate finite frame property with respect to Visser frames for several sublogics of $\IL$. 
Table~\ref{Tab:FMP} summarizes the results obtained in this paper on the soundness and completeness, and the finite frame property of these sublogics with respect to Visser frames. 
In particular, Theorem \ref{SV2} answers Problem~\ref{P1} affirmatively.

\begin{table}[ht]
\caption{Soundness, completeness, and finite frame property with respect to Visser frames for sublogics of $\IL$}\label{Tab:FMP}
\centering
\begin{tabular}[t]{|l|c|c|}
\hline
 & Soundness and completeness & finite frame property \\
\hline
\hline
$\IL^-(\J{4}_{+})$ & \checkmark (Theorem~\ref{SV})  & \checkmark (Theorem~\ref{SV})\\
\hline
$\IL^-(\J{1}, \J{4}_{+})$ & \checkmark (Theorem~\ref{SV}) & \checkmark (Theorem~\ref{SV})\\
\hline
$\IL^-(\J{4}_{+}, \J{5})$ & \checkmark (Theorem~\ref{SV}) & \checkmark (Theorem~\ref{SV})\\
\hline
$\IL^-(\J{1}, \J{4}_{+}, \J{5})$ & \checkmark (Theorem~\ref{SV}) & \checkmark (Theorem~\ref{SV})\\
\hline
$\IL^-(\J{2}_{+})$ & \checkmark (Theorem~\ref{SV2})& \checkmark (Theorem~\ref{SV2})\\
\hline
$\CL$ & \checkmark (Ignatiev~\cite{Ign91}) & \checkmark (Theorem~\ref{SV2})\\
\hline
$\IL^-(\J{2}_{+}, \J{5})$ & \checkmark (Theorem~\ref{SVIL}) & $\times$ (Proposition~\ref{Cex})\\
\hline
$\IL$ & \checkmark (Visser~\cite{Vis88}) & $\times$ (Visser~\cite{Vis97})\\
\hline
\end{tabular}
\end{table}


\section{Preliminaries}

In this section, we introduce the logic $\IL$, sublogics of $\IL$, and two relational semantics: Veltman frames and Visser frames. 
Then, we introduce and prove several basic properties for these logics and frames.

The language of interpretability logics consists of propositional variables $p, q, r,\cdots$, the logical constants $\top$, $\bot$, the connectives $\to$, $\lor$, $\land$, and the modal operators $\Box$ and $\rhd$. 
The modal formulas are given by 
\[
A::= \top \mid \bot \mid p \mid A \circ A \mid \Box A \mid A \rhd A
\]
where $\circ \in \{\to, \lor, \land\}$. We define $\Diamond A :\equiv \lnot \Box \lnot A$.  
The binary modal operator $\rhd$ binds stronger than $\to$, and weaker than $\lnot$, $\land$, $\lor$, $\Box$, and $\Diamond$. 

\begin{defn}
The logic $\IL$ of interpretability is axiomatized by the following axioms and rules: 
\begin{description}
	\item[$\G{0}$:] Propositional tautologies; 
	\item[$\G{1}$:] $\Box(A \to B) \to (\Box A \to \Box B)$; 
	\item[$\G{2}$:] $\Box (\Box A \to A) \to \Box A$;
	\item[$\J{1}$:] $\Box(A \to B) \to A \rhd B$;
	\item[$\J{2}$:] $(A \rhd B) \land (B \rhd C) \to A \rhd C$;
	\item[$\J{3}$:] $(A \rhd C) \land (B \rhd C) \to (A \lor B) \rhd C$;
	\item[$\J{4}$:] $A \rhd B \to (\Diamond A \to \Diamond B)$;
	\item[$\J{5}$:] $\Diamond A \rhd A$;
	\item[Modus Ponens:] $\dfrac{A \quad A \to B}{B}$;
	\item[Necessitation:] $\dfrac{A}{\Box A}$.
\end{description}
\end{defn}

\begin{defn}
The logic $\CL$ of conservativity is obtained by removing $\J{5}$ from $\IL$. 
\end{defn}

\begin{defn}
The logic $\IL^-$ is axiomatized by the axioms $\G{0}$, $\G{1}$, $\G{2}$, $\J{3}$, and 
\begin{description}
	\item[$\J{6}$:] $\Box \lnot A \leftrightarrow A \rhd \bot$, 
\end{description}
and the rules Modus Ponens, Necessitation, 
\begin{description}
	\item [$\R{1}$:] $\dfrac{A \to B}{C \rhd A \to C \rhd B}$, and
    \item [$\R{2}$:] $\dfrac{A \to B}{B \rhd C \to A \rhd C}$.
\end{description}
\end{defn}

The authors introduced in \cite{KO20} the logic $\IL^-$ and proved that $\IL^-$ is sound and complete with respect to the class of all (finite) $\IL^- \!$-frames. 

\begin{defn}
A triple ${\mathcal F} = (W, R, \{S_{x}\}_{x \in W})$ is said to be an \textit{$\IL^- \!$-frame} if the following conditions hold: 
\begin{itemize}
	\item $W$ is a non-empty set; 
	\item $R$ is a transitive and conversely well-founded binary relation on $W$; 
	\item For each $x \in W$, $S_{x} \subseteq R[x] \times W$, where $R[x] := \{y \in W \mid xRy\}$. 
\end{itemize}
We say that ${\mathcal F}$ is finite if $W$ is a finite set. 
\end{defn}

\begin{defn}\label{Def:VM}
A quadruple $(W, R, \{S_{x}\}_{x \in W}, \Vdash)$ is said to be an \textit{$\IL^- \!$-model} if the following conditions hold: 
\begin{itemize}
	\item $(W, R, \{S_{x}\}_{x \in W})$ is an $\IL^- \!$-frame. 
	\item $\Vdash$ is a usual satisfaction relation between $W$ and the set of all modal formulas fulfilling the following clauses: 
\begin{itemize}
	\item $x \Vdash \Box A \iff (\forall y \in W)(x R y \Rightarrow y \Vdash A)$;
	\item $x \Vdash A \rhd B \iff (\forall y \in W)\bigl(x R y \, \& \, y \Vdash A \Rightarrow \exists z(y S_{x} z \, \& \, z \Vdash B)\bigr)$.
\end{itemize}
\end{itemize}
Let $A$ be any modal formula and let ${\mathcal F} = (W, R, \{S_{x}\}_{x \in W})$ be any $\IL^- \!$-frame. 
We say that $A$ is \textit{valid} in ${\mathcal F}$ (${\mathcal F} \models A$) if for any $\IL^- \!$-model $(W, R, \{S_{x}\}_{x \in W}, \Vdash)$ and $x \in W$, $x \Vdash A$. 
\end{defn}

Here, we introduce the following axioms $\J{2}_+$ and $\J{4}_+$ which were introduced by Kurahashi and Okawa~\cite{KO20} and Visser~\cite{Vis88}, respectively.
\begin{description}
	\item[$\J{2}_+$:] $\bigl(A \rhd (B \lor C) \bigr) \land B \rhd C \to A \rhd C$.
	\item[$\J{4}_+$:] $\Box (A \to B) \to (C \rhd A \to C \rhd B)$.
\end{description}
Let $L(\Sigma_1, \ldots, \Sigma_n)$ denote the logic obtained by adding the axioms $\Sigma_1, \ldots, \Sigma_n$ to the logic $L$. 
Then, the following facts hold (for a proof, see~\cite{KO20}). 

\begin{fact}\label{Fact:24}
\noindent
\begin{enumerate}
	\item $\IL^-(\J{2}_{+}) \vdash \J{2}$. 
	\item $\IL^-(\J{1}, \J{2}) \vdash \J{2}_{+}$. 
	\item $\IL^-(\J{2}_{+}) \vdash \J{4}_{+}$. 
	\item $\IL^-(\J{4}_{+}) \vdash \J{4}$. 
	\item $\IL^-(\J{1}, \J{4}) \vdash \J{4}_{+}$. 
	\item The logics $\IL$ and $\CL$ are deductively equivalent to $\IL^-(\J{1}, \J{2}, \J{5})$ and $\IL^-(\J{1}, \J{2})$, respectively
\end{enumerate}
\end{fact}

\begin{fact}\label{FC}
Let ${\mathcal F} = (W, R, \{S_{x}\}_{x \in W})$ be any $\IL^- \!$-frame. 
\begin{enumerate}
	\item ${\mathcal F} \models \J{1}$ $\iff$ $(\forall x, y \in W)(x R y \Rightarrow yS_{x}y)$.
	\item ${\mathcal F} \models \J{2}_{+}$ $\iff$ ${\mathcal F} \models \J{4}_{+} \, \& \, (\forall x, y, z, v \in W)(yS_x z \, \& \, zS_{x} v \Rightarrow y S_{x} v)$.  
	\item ${\mathcal F} \models \J{4}_{+}$ $\iff$ $(\forall x, y, z \in W)(y S_x z \Rightarrow x R z)$.
	\item ${\mathcal F} \models \J{5}$ $\iff$ $(\forall x, y, z \in W)(x R y \, \& \, y R z \Rightarrow yS_{x}z)$.
\end{enumerate}
\end{fact}

\begin{fact}\label{SoIC}
Let $L$ be one of logics appearing in Figure~\ref{Comp}. 
Then the following are equivalent: 
\begin{enumerate}
	\item $L \vdash A$. 
	\item $A$ is valid in any (finite) $\IL^- \!$-frame ${\mathcal F}$ in which all axioms of $L$ are valid. 
\end{enumerate}
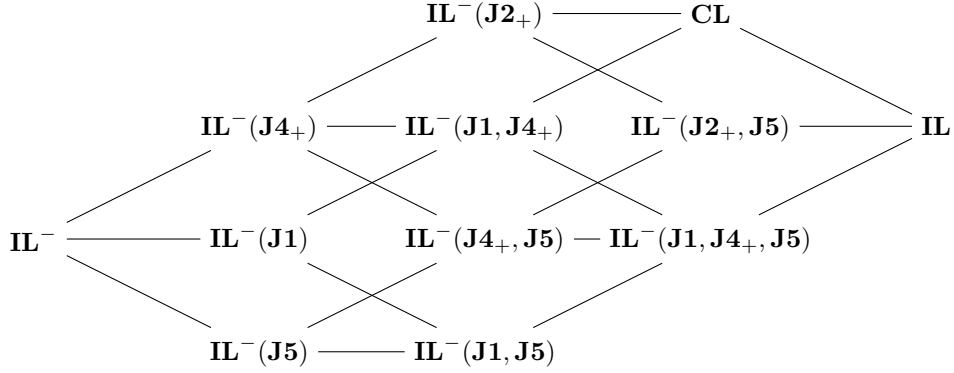
\begin{figure}[h]
\centering
\begin{tikzpicture}
\node (IL-) at (0,1.5) {$\IL^-$};
\node (IL5) at (3,0) {$\IL^-(\J{5})$};
\node (IL1) at (3,1.5) {$\IL^-(\J{1})$};
\node (IL4) at (3,3){$\IL^-(\J{4}_+)$};
\node (IL15) at (6,0){$\IL^-(\J{1}, \J{5})$};
\node (IL45) at (6,1.5){$\IL^-(\J{4}_+, \J{5})$};
\node (IL14) at (6,3){$\IL^-(\J{1}, \J{4}_+)$};
\node (IL2) at (6,4.5){$\IL^-(\J{2}_+)$};
\node (IL145) at (9,1.5){$\IL^-(\J{1}, \J{4}_+, \J{5})$};
\node (IL25) at (9,3){$\IL^-(\J{2}_+, \J{5})$};
\node (CL) at (9,4.5){$\CL$};
\node (IL) at (12,3){$\IL$};
\draw [-] (IL5)--(IL-);
\draw [-] (IL1)--(IL-);
\draw [-] (IL4)--(IL-);
\draw [-] (IL15)--(IL5);
\draw [-] (IL45)--(IL5);
\draw [-] (IL15)--(IL1);
\draw [-] (IL14)--(IL1);
\draw [-] (IL45)--(IL4);
\draw [-] (IL14)--(IL4);
\draw [-] (IL2)--(IL4);
\draw [-] (IL145)--(IL15);
\draw [-] (IL145)--(IL45);
\draw [-] (IL25)--(IL45);
\draw [-] (IL145)--(IL14);
\draw [-] (CL)--(IL14);
\draw [-] (IL25)--(IL2);
\draw [-] (CL)--(IL2);
\draw [-] (IL)--(IL145);
\draw [-] (IL)--(IL25);
\draw [-] (IL)--(CL);
\end{tikzpicture}
\caption{Sublogics of $\IL$ complete with respect to a corresponding class of $\IL^- \!$-frames}
\label{Comp}
\end{figure}
Here, the logics in Figure~\ref{Comp} extend as it goes to the right along each of lines. 
\end{fact}

\begin{rem}
From Fact~\ref{Fact:24}, for any extension $L$ of $\IL^-(\J{1})$ closed under the rules of $\IL^-$, we always identify the axioms $\J{2}_{+}$ and $\J{4}_{+}$ with $\J{2}$ and $\J{4}$, respectively. 
Also, we always identify $\IL$ and $\CL$ with $\IL^-(\J{1}, \J{2}, \J{5})$ and $\IL^-(\J{1}, \J{2})$, respectively.  
\end{rem}

In the next section, we investigate the modal completeness of several sublogics of $\IL$ in Figure~\ref{Comp} with respect to Visser frames which were originally introduced by Visser~\cite{Vis88}. 
In this paper, we consider the following \textit{$\IL^-(\J{4}_{+})$-Visser frames} as the basis, which are general notion of Visser's Visser frames. 

\begin{defn}
A triple $(W, R, S)$ is said to be a \textit{$\IL^-(\J{4}_{+})$-Visser frame} if the following conditions hold: 
\begin{itemize}
	\item $W$ is a non-empty set; 
	\item $R$ is a transitive and conversely well-founded binary relation on $W$; 
	\item $S$ is a binary relation on $W$. 
\end{itemize}
\end{defn}

\begin{defn}\label{Def:SVM}
A quadruple $(W, R, S, \Vdash)$ is said to be a \textit{$\IL^- (\J{4}_{+})$-Visser model} if the following conditions hold: 
\begin{itemize}
	\item $(W, R, S)$ is an $\IL^- (\J{4}_{+})$-Visser frame; 
	\item As in Definition~\ref{Def:VM}, we define a satisfaction relation $\Vdash$ with the following clause:
\[
x \Vdash A \rhd B \iff (\forall y \in W)\bigl(x R y \, \& \, y \Vdash A \Rightarrow (\exists z \in W)(xRz \, \& \, ySz \, \& \, z \Vdash B)\bigr).
\]
\end{itemize}
For any modal formula $A$ and $\IL^-(\J{4}_{+})$-Visser frame ${\mathcal F}$, we define ${\mathcal F} \models A$ as in Definition~\ref{Def:VM}. 
\end{defn}

\begin{rem}\label{Rem:J4}
For an $\IL^- (\J{4}_{+})$-Visser model $(W, R, S, \Vdash)$, it may be natural to consider a satisfaction relation with the following clause:
\[
x \Vdash A \rhd B \iff (\forall y \in W)\bigl(x R y \, \& \, y \Vdash A \Rightarrow (\exists z \in W)(ySz \, \& \, z \Vdash B)\bigr).
\]
However, if we adopt this definition, then the persistency axiom 
\begin{description}
	\item[$\mathbf{P}$:] $A \rhd B \to \Box(A \rhd B)$
\end{description}
is valid in all Visser frames. 
In fact, the authors with Iwata proved in \cite[Theorem 4.13]{IKO22} that the logic $\IL^-(\mathbf{P})$ is sound and complete with respect to the class of all (finite) Visser frames in this definition. 
Therefore, in this paper, we adopt the definition in Definition~\ref{Def:SVM} which was originally adopted by Visser in~\cite{Vis88}. 
On the other hand, our definition causes that $\J{4}_{+}$  is valid in all $\IL^-(\J{4}_{+})$-Visser frames (see Proposition~\ref{SFC}.2). 
Therefore, we focus on and investigate extensions of $\IL^-(\J{4}_{+})$. 
\end{rem}

\begin{prop}\label{SFC}
Let ${\mathcal F} = (W, R, S)$ be any $\IL^-(\J{4}_{+})$-Visser frame. 
\begin{enumerate}
	\item ${\mathcal F} \models \J{3}$. 
	\item ${\mathcal F} \models \J{4}_{+}$. 
	\item ${\mathcal F} \models \J{6}$. 
	\item ${\mathcal F} \models \J{1}$ $\iff$ $(\forall x, y \in W)(x R y \Rightarrow ySy)$. 
	\item ${\mathcal F} \models \J{2}_{+}$ $\iff$ $(\forall x \in W)(\forall y, z, v \in R[x])(y S z \, \& \, z S v \Rightarrow y S v)$. 
	\item ${\mathcal F} \models \J{5}$ $\iff$ $(\forall x, y, z \in W)(x R y \, \& \, y R z \Rightarrow ySz)$. 
\end{enumerate}
\end{prop}

\begin{proof}
1. Suppose $x \Vdash (A \rhd C) \land (B \rhd C)$. 
We show $x \Vdash (A \lor B) \rhd C$. 
Let $y \in W$ be such that $x R y$ and $y \Vdash A \lor B$. 
If $y \Vdash A$, then there exists $z \in W$ such that $x R z$, $y S z$, and $z \Vdash C$ by $x \Vdash A \rhd C$. 
If $y \Vdash B$, then there exists $z \in W$ such that $x R z$, $y S z$, and $z \Vdash C$ by $x \Vdash B \rhd C$. 
In any case, we obtain $x \Vdash (A \lor B) \rhd C$. 

\medskip

2. Suppose $x \Vdash \Box (A \to B) \land (C \rhd A)$. 
We show $x \Vdash C \rhd B$. 
Let $y \in W$ be such that $x R y$ and $y \Vdash C$. 
Since $x \Vdash C \rhd A$, there exists $z \in W$ such that $x R z$, $y S z$, and $z \Vdash A$. 
Since $x \Vdash \Box (A \to B)$, we have $z \Vdash B$. 
Therefore, $x \Vdash C \rhd B$. 

\medskip

3. For $x \in W$, 
\begin{align*}
    x \Vdash \Box \lnot A & \iff (\forall y \in W) (x R y \Rightarrow  y \Vdash \neg A) \\
    & \iff (\forall y \in W) (x R y \ \&\ y \Vdash A \Rightarrow (\exists z \in W) (x R z\ \&\ y S z\ \&\  z \Vdash \bot)) \\
    & \iff x \Vdash A \rhd \bot. 
\end{align*}

\medskip

4. ($\Rightarrow$): 
Suppose $x R y$. 
Let $(W, R, S, \Vdash)$ be an $\IL^-(\J{4}_{+})$-Visser model such that  $a \Vdash p \iff a = y$ and $a \Vdash q \iff a = y$. 
Then, we have $x \Vdash \Box(p \to q)$. 
Thus, $x \Vdash p \rhd q$ by $\J{1}$. 
Since $x R y$ and $y \Vdash p$, there exists $y' \in W$ such that $x R y'$, $y S y'$, and $y' \Vdash q$. 
Then, $y' = y$, and hence $y S y$. 

($\Leftarrow$): Suppose $x \Vdash \Box(A \to B)$. 
We show $x \Vdash A \rhd B$. 
Let $y \in W$ be with $x R y$ and $y \Vdash A$. 
Since $y \Vdash B$ and $y S y$, we obtain $x \Vdash A \rhd B$. 

\medskip

5. ($\Rightarrow$): 
Suppose $y, z, v \in R[x]$, $y S z$, and $z S v$. 
Let $(W, R, S, \Vdash)$ be an $\IL^-(\J{4}_{+})$-Visser model such that $a \Vdash p \iff a = y$, $a \Vdash q \iff a = z$, and $a \Vdash r \iff a = v$. 
To show $x \Vdash p \rhd (q \lor r)$, let $y' \in W$ be with $x R y'$ and $y' \Vdash p$. 
Then, $y' = y$. 
Since $x R z$, $y S z$, and $z \Vdash q \lor r$, we have $x \Vdash p \rhd (q \lor r)$. 
To show $x \Vdash q \rhd r$, let $z' \in W$ be with $x R z'$ and $z' \Vdash q$. 
Then, $z' = z$. 
Since $x R v$, $z S v$, and $v \Vdash r$, we have $x \Vdash q \rhd r$. 
Therefore, by $\J{2}_+$, we get $x \Vdash p \rhd r$. 
Since $x R y$ and $y \Vdash p$, there exists $v' \in W$ such that $x R v'$, $y S v'$, and $v' \Vdash r$. 
Since $v' = v$, we conclude $y S v$. 

($\Leftarrow$): 
Suppose $x \Vdash A \rhd (B \lor C)$ and $x \Vdash B \rhd C$. 
We show $x \Vdash A \rhd C$. 
Let $y \in W$ be such that $x R y$ and $y \Vdash A$. 
Since $x \Vdash A \rhd (B \lor C)$, there exists $z \in W$ such that $x R z$, $y S z$, and $z \Vdash B \lor C$. 
If $z \Vdash C$, then we are done. 
Assume $z \Vdash B$. 
Since $x \Vdash B \rhd C$ and $x R z$, there exists $v \in W$ such that $x R v$, $z S v$, and $v \Vdash C$. 
Since $y, z, v \in R[x]$, $y S z$, and $z S v$, we obtain $y S v$. 
Therefore, we obtain $x \Vdash A \rhd C$. 

\medskip

6. ($\Rightarrow$): Suppose $x R y$ and $y R z$. 
Let $(W, R, S, \Vdash)$ be an $\IL^-(\J{4}_{+})$-Visser model such that $a \Vdash p \iff a = z$. 
Then, we have $y \Vdash \Diamond p$. 
Since $x \Vdash \Diamond p \rhd p$ and $x R y$, there exists $z' \in W$ such that $x R z'$, $y S z'$, and $z' \Vdash p$. 
Then, $z' = z$, and hence $y S z$. 

($\Leftarrow$): Let $x \in W$. 
We show $x \Vdash \Diamond A \rhd A$. 
Let $y \in W$ be such that $x R y$ and $y \Vdash \Diamond A$. 
Then, there exists $z \in W$ such that $y R z$ and $z \Vdash A$. 
Since $x R y$ and $y R z$, we have $x R z$ and $y S z$. 
Therefore, we conclude $x \Vdash \Diamond A \rhd A$. 
\end{proof}

\begin{defn}\label{L-frame}
Let $L$ be one of extensions of $\IL^-(\J{4}_{+})$ appearing in Figure~\ref{Comp}, and let $\mathcal{F} = (W, R, S)$ be any $\IL^-(\J{4}_{+})$-Visser frame. 
We say that \textit{$\mathcal{F}$ is a $L$-Visser frame} if the following conditions hold:
\begin{itemize}
	\item If $\J{1}$ is an axiom of $L$, then $S$ is reflexive;
	\item If $\J{2}_{+}$ is an axiom of $L$, then $S$ is transitive;
	\item If $\J{5}$ is an axiom of $L$, then $R \subseteq S$.
\end{itemize}
\end{defn}

Proposition \ref{SFC} shows that each logic $L$ as in Definition \ref{L-frame} is sound with respect to all $L$-Visser frames, that is, every $L$-Visser frame validates all axioms of $L$. 
Conversely, the modal completeness of the logics $\IL$ and $\CL$ with respect to corresponding Visser frames is already known. 
On the other hand, Visser~\cite{Vis97} proved that $\IL$ does not have finite frame property. 

\begin{fact}\label{Fact:SComp}
Let $A$ be any modal formula. 
\begin{enumerate}
	\item \textup{(Visser~\cite[pp.~18--21]{Vis88})} $\IL \vdash A \iff$ $A$ is valid in all $\IL$-Visser frames. 
	\item \textup{(Visser~\cite[pp.~328--329]{Vis97})} $\IL$ does not have finite frame property with respect to $\IL$-Visser frames. 
	\item \textup{(Ignatiev~\cite[Theorem 5]{Ign91})} $\CL \vdash A \iff$ $A$ is valid in all $\CL$-Visser frames. 
\end{enumerate}
\end{fact}

Our main purpose of this paper is to prove the finite frame property for $\CL$ with respect to Visser frames (Theorem~\ref{SV2}).

\section{Modal completeness and finite frame property}

In this section, we study modal completeness of each extension of $\IL^-(\J{4}_{+})$ with respect to Visser frames. 
Among other things, we prove that the logic $\CL$ has finite frame property. 

\subsection{The finite frame property of $\IL^-(\J{4}_{+})$, $\IL^-(\J{1}, \J{4}_{+})$, $\IL^-(\J{4}_{+}, \J{5})$, and $\IL^-(\J{1}, \J{4}_{+}, \J{5})$}

Here we investigate the logics $\IL^-(\J{4}_{+})$, $\IL^-(\J{1}, \J{4}_{+})$, $\IL^-(\J{4}_{+}, \J{5})$, and $\IL^-(\J{1}, \J{4}_{+}, \J{5})$. 

\begin{thm}\label{SV}
Let $L$ be either $\IL^-(\J{4}_{+})$, $\IL^-(\J{1}, \J{4}_{+})$, $\IL^-(\J{4}_{+}, \J{5})$, or $\IL^-(\J{1}, \J{4}_{+}, \J{5})$.
For any modal formula $A$, the following are equivalent:
\begin{enumerate}
	\item $L \vdash A$. 
	\item $A$ is valid in all (finite) $L$-Visser frames.
\end{enumerate}
\end{thm}

\begin{proof}
($1 \Rightarrow 2$): Let $\mathcal{F}$ be any $L$-Visser frame. 
By the definition of $\mathcal{F}$ and Proposition~\ref{SFC}, all axioms of $L$ are valid in $\mathcal{F}$. 
Also, the validity of formulas in $\mathcal{F}$ is preserved by every rule of $\IL^-$.

($2 \Rightarrow 1$): Suppose $L \nvdash A$. 
Then, there exist an $\IL^-\!$-model $(W, R, \{S_{x}\}_{x \in W}, \mathord \Vdash)$ and $w \in W$ such that $w \nVdash A$ and all axioms of $L$ are valid in $(W, R, \{S_{x}\}_{x \in W})$ by Fact~\ref{SoIC}. 
We define the following set $W'$ of sequences of elements of $W$. 
\[
W' := \{\seq{x_{0},\ldots,x_{n}} \mid n \geq 0, \, x_0, \ldots, x_n \in W, \, \text{and} \,  \forall j < n\, (x_{j} R x_{j+1})\}.
\]
Since $R$ is transitive and irreflexive, the length of each sequence of $W'$ is less than or equal to the number of elements of $W$. 
Therefore, $W'$ is a finite set. 

Let $\varepsilon$ denote the empty sequence, and let $\mathbf{x}, \mathbf{y}, \mathbf{z}, \ldots$ denote elements of $W' \cup \{\varepsilon\}$.

\begin{itemize}
 	\item $\mathbf{x} \sqsubseteq \mathbf{y} :\iff$ $\mathbf{x}$ is an initial segment of $\mathbf{y}$. 
	\item $\mathbf{x} \sqsubset \mathbf{y} :\iff$ $\mathbf{x} \sqsubseteq \mathbf{y}$ and $\mathbf{x} \neq \mathbf{y}$. 
	\item $\mathbf{y} \cap \mathbf{z}$ is the longest common initial segment of $\mathbf{y}$ and $\mathbf{z}$. 
    That is, $\mathbf{y} \cap \mathbf{z}$ is the maximum element of $\{\mathbf{w} \in W' \cup \{\varepsilon\} \mid \mathbf{w} \sqsubseteq \mathbf{y} \, \& \, \mathbf{w} \sqsubseteq \mathbf{z}\}$ with respect to $\sqsubseteq$. 
\end{itemize}

Let $\mathbf{x} = \seq{x_{0},\ldots,x_{n}}$ be an element of $W'$. 
\begin{itemize}
	\item $\pre(\mathbf{x}) := \left\{
\begin{array}{ll}
\seq{x_{0},\ldots,x_{n-1}} & \text{if}\ n \geq 1,\\
\varepsilon & \text{if}\ n = 0.
\end{array}
\right.
$
	\item $\mathbf{x}^{e} := x_{n}$. 
	\item For each $x \in W$, let $\mathbf{x} * \seq{x} := \seq{x_{0},\ldots,x_{n}, x}$. 
\end{itemize}


Then, we define the binary relations $R'$ and $S'$ on $W'$ and the satisfaction relation $\Vdash'$ as follows:
\begin{itemize}
	\item $\mathbf{x} R' \mathbf{y} :\iff \mathbf{x} \sqsubset \mathbf{y}$; 
	\item $\mathbf{y} S' \mathbf{z} :\iff$ For $\mathbf{x} := \pre(\mathbf{y}) \cap \pre(\mathbf{z}) \in W' \cup \{\varepsilon\}$, either ($\mathbf{x} = \varepsilon$ and $\mathbf{y} \sqsubseteq \mathbf{z}$) or ($\mathbf{x} \neq \varepsilon$ and $\mathbf{y}^e S_{\mathbf{x}^e} \mathbf{z}^e$). 
	\item $\mathbf{x} \Vdash' p :\iff \mathbf{x}^e \Vdash p$. 
\end{itemize}
Since $R'$ is obviously transitive and irreflexive, $\mathcal{F} = (W', R', S')$ is a finite $\IL^-(\J{4}_{+})$-Visser frame. 
We show that $\mathcal{F}$ is a finite $L$-Visser frame. 
\begin{itemize}
	\item The case that $\J{1}$ is an axiom of $L$: we show that $S'$ is reflexive. 
For any $\mathbf{y} \in W'$, let $\mathbf{x} : = \pre(\mathbf{y}) = \pre(\mathbf{y}) \cap \pre(\mathbf{y})$. 
If $\mathbf{x} = \varepsilon$, then we obtain $\mathbf{y} S' \mathbf{y}$ since $\mathbf{y} \sqsubseteq \mathbf{y}$.  
If $\mathbf{x} \neq \varepsilon$, then ${\mathbf{x}}^{e} R \mathbf{y}^{e}$. 
We have $\mathbf{y}^{e} S_{{\mathbf{x}}^{e}} \mathbf{y}^{e}$ by Fact~\ref{FC}.1. 
Therefore, we obtain $\mathbf{y} S' \mathbf{y}$.

	\item The case that $\J{5}$ is an axiom of $L$: we show that $R' \subseteq S'$. 
Suppose $\mathbf{y} R' \mathbf{z}$. 
Let $\mathbf{x} : = \pre(\mathbf{y}) \cap \pre(\mathbf{z})$. 
Since $\pre(\mathbf{y}) \sqsubset \pre(\mathbf{z})$, we have $\mathbf{x} = \pre(\mathbf{y})$. 
If $\mathbf{x} = \varepsilon$, then we obtain $\mathbf{y} S' \mathbf{z}$ because $\mathbf{y} \sqsubseteq \mathbf{z}$.  
If $\mathbf{x} \neq \varepsilon$, then we have ${\mathbf{x}}^{e} R \mathbf{y}^{e}$ and $\mathbf{y}^{e} R \mathbf{z}^{e}$. 
By Fact~\ref{FC}.4, we get $\mathbf{y}^{e} S_{{\mathbf{x}}^{e}} \mathbf{z}^{e}$. 
Therefore, we conclude $\mathbf{y} S' \mathbf{z}$. 
\end{itemize}

Following claim is a key of the proof. 

\begin{cl}\label{cl1}
For any modal formula $D$ and $\mathbf{x} \in W'$, 
\[
\mathbf{x} \Vdash' D \iff \mathbf{x}^e \Vdash D. 
\]
\end{cl}

\begin{proof}
We prove the claim by induction on the construction of $D$. 
We only prove the case $D \equiv B \rhd C$. 

\medskip

($\Rightarrow$): 
Assume $\mathbf{x} \Vdash' B \rhd C$. 
We show $\mathbf{x}^e \Vdash B \rhd C$. 
Let $y \in W$ be such that $\mathbf{x}^e R y$ and $y \Vdash B$. 
Let $\mathbf{y}:= \mathbf{x} * \seq{y}$. 
Then, we have $\mathbf{y} \in W'$ and $\mathbf{x} R' \mathbf{y}$. 
Since $y = \mathbf{y}^e$, by the induction hypothesis, $\mathbf{y} \Vdash' B$. 
By our assumption, there exists $\mathbf{z} \in W'$ such that $\mathbf{x} R' \mathbf{z}$, $\mathbf{y} S' \mathbf{z}$, and $\mathbf{z} \Vdash' C$. 
Then, we get $\pre(\mathbf{y}) \cap \pre(\mathbf{z}) = \mathbf{x} \neq \varepsilon$ because $\mathbf{x} = \pre(\mathbf{y})$ and $\mathbf{x} \sqsubseteq \pre(\mathbf{z})$. 
Then, $\mathbf{y}^e S_{\mathbf{x}^e} \mathbf{z}^e$ follows from $\mathbf{y} S' \mathbf{z}$ by the definition of $S'$. 
By the induction hypothesis, we have $\mathbf{z}^e \Vdash C$, and hence we obtain $\mathbf{x}^e \Vdash B \rhd C$. 

\medskip

($\Leftarrow$): 
Assume $\mathbf{x}^e \Vdash B \rhd C$. 
We show $\mathbf{x} \Vdash' B \rhd C$. 
Let $\mathbf{y} \in W'$ be such that $\mathbf{x} R' \mathbf{y}$ and $\mathbf{y} \Vdash' B$. 
Then, we have $\mathbf{x}^e R \mathbf{y}^e$ and $\mathbf{y}^e \Vdash B$ by the induction hypothesis. 
By our assumption, there exists $z \in W$ such that $\mathbf{y}^e S_{\mathbf{x}^e} z$ and $z \Vdash C$. 
Let $\mathbf{z} := \mathbf{x}*\seq{z}$. 
By Fact~\ref{FC}.3, $\mathbf{x}^{e} R z$, and hence we have $\mathbf{z} \in W'$ and $\mathbf{x} R' \mathbf{z}$. 
Then, $\pre(\mathbf{y}) \cap \pre(\mathbf{z}) = \mathbf{x} \neq \varepsilon$ because $\mathbf{x} = \pre(\mathbf{z})$ and $\mathbf{x} \sqsubseteq \pre(\mathbf{y})$. 
Therefore, we obtain $\mathbf{y} S' \mathbf{z}$. 
Since $\mathbf{z} \Vdash C$ by the induction hypothesis, we conclude $\mathbf{x} \Vdash' B \rhd C$. 
\end{proof}

Since $w \nVdash A$, we obtain $\seq{w} \nVdash' A$ by the claim.  
Thus, $A$ is not valid in $\mathcal{F}$. 
\end{proof}

\subsection{The finite frame property of $\IL^-(\J{2}_{+})$ and $\CL$}

In this subsection, we show that the logics $\IL^-(\J{2}_{+})$ and $\CL$ have finite frame property with respect to Visser frames. 
This answers Problem~\ref{P1} affirmatively. 
We achieved this by constructing a finite Visser frame similar to the one we constructed in the proof of Theorem \ref{SV}, unlike the method of proof due to Ignatiev. 
However, we need to modify the construction of a frame given in the proof of Theorem \ref{SV} because the transitivity of $S'$ is not guaranteed as it is.

\begin{thm}\label{SV2}
Let $L$ be either $\IL^-(\J{2}_{+})$ or $\CL$. 
For any modal formula $A$, the following are equivalent:
\begin{enumerate}
	\item $L \vdash A$. 
	\item $A$ is valid in all (finite) $L$-Visser frames. 
\end{enumerate}
\end{thm}

\begin{proof}
We only prove the implication ($2 \Rightarrow 1$). 
Suppose $L \nvdash A$. 
By Fact~\ref{SoIC}, there exist a finite $\IL^-\!$-model $(W, R, \{S_{x}\}_{x \in W}, \Vdash)$ and $w \in W$ such that $w \nVdash A$ and all axioms of $L$ are valid in the frame $(W, R, \{S_{x}\}_{x \in W})$. 
Then, we define $W'$, $R'$, and $\Vdash'$ as in the proof of Theorem~\ref{SV}. 
For $\mathbf{x}, \mathbf{y} \in W' \cup \{\varepsilon\}$ with $\mathbf{x} \sqsubseteq \mathbf{y}$, let $\mathbf{y} - \mathbf{x}$ be the set of all elements of $W$ which appears in $\mathbf{y}$ but not in $\mathbf{x}$. 
We define the binary relation $S'$ on $W'$ as follows: 
\begin{itemize}
	\item $\mathbf{y} S' \mathbf{z} :\iff$ For $\mathbf{x} := \pre(\mathbf{y}) \cap \pre(\mathbf{z})$, either $(\mathbf{x} = \varepsilon$ and $\mathbf{y} = \mathbf{z})$ or the following three conditions hold: 
\begin{enumerate}
	\item $\mathbf{x} \neq \varepsilon$;
	\item $\pre(\mathbf{y}) - \pre(\mathbf{x}) \supseteq \pre(\mathbf{z}) - \pre(\mathbf{x})$; \hfill($\dagger$)
	\item there exist $l \geq 0$, $v_{0},\ldots,v_{l} \in \pre(\mathbf{y}) - \pre(\mathbf{x})$, and $a_{0},\ldots,a_{l+1} \in W$ such that $a_{0} = \mathbf{y}^{e}$, $a_{l+1} = \mathbf{z}^{e}$, and $a_{i} S_{v_{i}} a_{i+1}$ for all $i \leq l$. \hfill($\ddagger$)
\end{enumerate}
\end{itemize}

Then, $\mathcal{F} = (W', R', S')$ is a finite $\IL^-(\J{4}_{+})$-Visser frame. 
We show that $L$ is a finite $L$-Visser frame. 

\begin{itemize}
	\item In the case of $L = \CL$, we show that $S'$ is reflexive. 
Let $\mathbf{y} \in W'$ and let $\mathbf{x} : = \pre(\mathbf{y}) \cap \pre(\mathbf{y}) = \pre(\mathbf{y})$. 
If $\mathbf{x} = \varepsilon$, we obtain $\mathbf{y} S' \mathbf{y}$ because $\mathbf{y} = \mathbf{y}$. 
If $\mathbf{x} \neq \varepsilon$, then $\mathbf{x}^e \in \pre(\mathbf{y}) - \pre(\mathbf{x})$. 
Since $\pre(\mathbf{y}) - \pre(\mathbf{x}) \supseteq \pre(\mathbf{y}) - \pre(\mathbf{x})$ and $\mathbf{y}^{e} S_{{\mathbf{x}}^{e}} \mathbf{y}^{e}$ by Fact~\ref{FC}.1, we obtain $\mathbf{y} S' \mathbf{y}$.

	\item We show that $S'$ is transitive. 
Suppose $\mathbf{y} S' \mathbf{z}$ and $\mathbf{z} S' \mathbf{v}$, and let
\begin{itemize}
    \item $\mathbf{x}_{1} : = \pre(\mathbf{y}) \cap \pre(\mathbf{z})$, 
    \item $\mathbf{x}_{2} : = \pre(\mathbf{z}) \cap \pre(\mathbf{v})$, 
    \item $\mathbf{x} : = \pre(\mathbf{y}) \cap \pre(\mathbf{v})$.
\end{itemize}
If $\mathbf{x}_{1} = \varepsilon$ or $\mathbf{x}_{2} = \varepsilon$, then $\mathbf{y} = \mathbf{z}$ or $\mathbf{z} = \mathbf{v}$, and hence we obtain $\mathbf{y} S' \mathbf{v}$. 
So, we may assume $\mathbf{x}_{1} \neq \varepsilon$ and $\mathbf{x}_{2} \neq \varepsilon$. 
Since $\mathbf{x}_{1}$ and $\mathbf{x}_{2}$ are initial segments of $\pre(\mathbf{z})$, there are three possibilities: $\mathbf{x}_{1} \sqsubset \mathbf{x}_{2}$, $\mathbf{x}_{2} \sqsubset \mathbf{x}_{1}$, and $\mathbf{x}_{1} = \mathbf{x}_{2}$. 
In either case, we have that $\mathbf{x}' := \mathbf{x}_{1} \cap \mathbf{x}_{2}$ equals to $\mathbf{x}_{1}$ or $\mathbf{x}_{2}$, and hence $\mathbf{x}' \neq \varepsilon$. 
Since $\mathbf{x}'$ is an initial segment of $\pre(\mathbf{y})$ and $\pre(\mathbf{v})$, we have $\mathbf{x}' \sqsubseteq \mathbf{x}$ by the maximality of $\mathbf{x}$. 
Thus, we get $\mathbf{x} \neq \varepsilon$. 
To show $\mathbf{y} S' \mathbf{v}$, we need to prove that $\mathbf{x}$ satisfies ($\dagger$) and ($\ddagger$). 
We distinguish the following three cases.

\subparagraph*{Case 1:} $\mathbf{x}_{1} \sqsubset \mathbf{x}_{2}$. 

At first, we show $\mathbf{x} = \mathbf{x}_1$. 
Since $\mathbf{x}_1 \sqsubseteq \pre(\mathbf{y})$ and $\mathbf{x}_1 \sqsubset \mathbf{x}_2 \sqsubseteq \pre(\mathbf{v})$, we have $\mathbf{x}_1 \sqsubseteq \mathbf{x}$ by the maximality of $\mathbf{x}$. 
On the other hand, since both $\mathbf{x}$ and $\mathbf{x}_2$ are initial segments of $\pre(\mathbf{v})$, we have $\mathbf{x} \sqsubseteq \mathbf{x}_2$ or $\mathbf{x}_2 \sqsubseteq \mathbf{x}$. 
If $\mathbf{x}_2 \sqsubseteq \mathbf{x}$, then we would have $\mathbf{x}_2 \sqsubseteq \mathbf{x} \sqsubseteq \pre(\mathbf{y})$. 
Since $\mathbf{x}_2 \sqsubseteq \pre(\mathbf{z})$, we would get $\mathbf{x}_2 \sqsubseteq \mathbf{x}_1$ by the maximality of $\mathbf{x}_1$, a contradiction. 
Therefore, we obtain $\mathbf{x} \sqsubseteq \mathbf{x}_2$. 
Since $\mathbf{x} \sqsubseteq \pre(\mathbf{y})$ and $\mathbf{x} \sqsubseteq \mathbf{x}_2 \sqsubseteq \pre(\mathbf{z})$, we have $\mathbf{x} \sqsubseteq \mathbf{x}_1$ by the maximality of $\mathbf{x}_1$ again. 
We have proved $\mathbf{x} = \mathbf{x}_1$. 
To conclude $\mathbf{y} S'  \mathbf{v}$, we verify the conditions ($\dagger$) and ($\ddagger$) as follows: 

($\dagger$): We obtain the inclusion $\pre(\mathbf{y}) - \pre(\mathbf{x}) \supseteq \pre(\mathbf{v}) - \pre(\mathbf{x})$ as follows: 
\begin{align*}
    \pre(\mathbf{y}) - \pre(\mathbf{x}) & = \pre(\mathbf{y}) - \pre(\mathbf{x}_1) \tag{$\mathbf{x} = \mathbf{x}_1$}\\
    & \supseteq \pre(\mathbf{z}) - \pre(\mathbf{x}_1) \tag{$\mathbf{y} S' \mathbf{z}$} \\
    & = (\pre(\mathbf{z}) - \pre(\mathbf{x}_2)) \cup (\pre(\mathbf{x}_2) - \pre(\mathbf{x}_1)) \tag{$\mathbf{x}_{1} \sqsubset \mathbf{x}_{2} \sqsubseteq \pre(\mathbf{z})$} \\
    & \supseteq (\pre(\mathbf{v}) - \pre(\mathbf{x}_2)) \cup (\pre(\mathbf{x}_2) - \pre(\mathbf{x}_1)) \tag{$\mathbf{z} S' \mathbf{v}$} \\
    & = \pre(\mathbf{v}) - \pre(\mathbf{x}_1) \tag{$\mathbf{x}_{1} \sqsubset \mathbf{x}_{2} \sqsubseteq \pre(\mathbf{v})$} \\
    & = \pre(\mathbf{v}) - \pre(\mathbf{x}). \tag{$\mathbf{x} = \mathbf{x}_1$}
\end{align*}


($\ddagger$): Since $\mathbf{y} S' \mathbf{z}$ and $\mathbf{z} S' \mathbf{v}$, 
\begin{itemize}
	\item there exist $w_{0},\ldots,w_{l} \in \pre(\mathbf{y}) - \pre(\mathbf{x}_{1})$ and $a_{0},\ldots,a_{l+1} \in W$ such that
$a_{0} = \mathbf{y}^{e}$, $a_{l+1} = \mathbf{z}^{e}$ and $a_{i} S_{w_{i}} a_{i+1}$ for all $i \leq l$, and 
	\item there exist $w'_{0},\ldots,w'_{l'} \in \pre(\mathbf{z}) - \pre(\mathbf{x}_{2})$ and $b_{0},\ldots,b_{l'+1} \in W$ such that $b_{0} = \mathbf{z}^{e}$, $b_{l'+1} = \mathbf{v}^{e}$ and $b_{i} S_{w_{i}'} b_{i+1}$ for all $i \leq l'$. 
\end{itemize}
We have $\pre(\mathbf{y}) - \pre(\mathbf{x}_1) \supseteq \pre(\mathbf{z}) - \pre(\mathbf{x}_{1}) \supseteq \pre(\mathbf{z}) - \pre(\mathbf{x}_{2})$.
So, $w'_{0},\ldots,w'_{l'}$ are also in $\pre(\mathbf{y}) - \pre(\mathbf{x}_{1})$. 
Then, $w_{0},\ldots,w_{l}, w'_{0}, \ldots, w'_{l'} \in \pre(\mathbf{y}) - \pre(\mathbf{x})$ and $\mathbf{y}^e = a_{0},\ldots,a_{l+1} = b_{0}, b_{1}, \ldots, b_{l'+1} = \mathbf{v}^e$ match the condition ($\ddagger$).


\subparagraph*{Case 2:} $\mathbf{x}_{2} \sqsubset \mathbf{x}_{1}$. 

As in Case 1, it is shown that $\mathbf{x} = \mathbf{x}_2$. 

($\dagger$): The inclusion $\pre(\mathbf{y}) - \pre(\mathbf{x}) \supseteq \pre(\mathbf{v}) - \pre(\mathbf{x})$ is obtained as follows: 
\begin{align*}
    \pre(\mathbf{y}) - \pre(\mathbf{x}) & = \pre(\mathbf{y}) - \pre(\mathbf{x}_2) \tag{$\mathbf{x} = \mathbf{x}_2$}\\
    & = (\pre(\mathbf{y}) - \pre(\mathbf{x}_1)) \cup (\pre(\mathbf{x}_1) - \pre(\mathbf{x}_2)) \tag{$\mathbf{x}_{2} \sqsubset \mathbf{x}_{1} \sqsubseteq \pre(\mathbf{y})$} \\
    & \supseteq (\pre(\mathbf{z}) - \pre(\mathbf{x}_1)) \cup (\pre(\mathbf{x}_1) - \pre(\mathbf{x}_2)) \tag{$\mathbf{y} S' \mathbf{z}$} \\
    & = \pre(\mathbf{z}) - \pre(\mathbf{x}_2) \tag{$\mathbf{x}_{2} \sqsubset \mathbf{x}_{1} \sqsubseteq \pre(\mathbf{z})$} \\
    & \supseteq \pre(\mathbf{v}) - \pre(\mathbf{x}_2) \tag{$\mathbf{z} S' \mathbf{v}$} \\
    & = \pre(\mathbf{v}) - \pre(\mathbf{x}). \tag{$\mathbf{x} = \mathbf{x}_2$}
\end{align*}


($\ddagger$): 
The inclusions $\pre(\mathbf{y}) - \pre(\mathbf{x}) \supseteq \pre(\mathbf{y}) - \pre(\mathbf{x}_1)$ and $\pre(\mathbf{y}) - \pre(\mathbf{x}) \supseteq \pre(\mathbf{z}) - \pre(\mathbf{x}_2)$ also hold in this case. 
These inclusions together with $\mathbf{y} S' \mathbf{z}$ and $\mathbf{z} S' \mathbf{v}$ guarantee the condition ($\ddagger$). 


\subparagraph*{Case 3:} $\mathbf{x}_{1} = \mathbf{x}_{2}$. 

We prove $\mathbf{x}_1 = \mathbf{x}_2 = \mathbf{x}$. 
We already showed that $\mathbf{x}_{1} \cap \mathbf{x}_{2} \sqsubseteq \mathbf{x}$, and hence $\mathbf{x}_1 = \mathbf{x}_2 \sqsubseteq \mathbf{x}$. 
Suppose, towards a contradiction, that $\mathbf{x}_2 \sqsubset \mathbf{x}$. 
Then, there exists $a \in W$ such that $\mathbf{x}_2 \ast \seq{a} \sqsubseteq \mathbf{x}$. 
Since $\mathbf{x}_2 \ast \seq{a} \sqsubseteq \mathbf{x} \sqsubseteq \pre(\mathbf{v})$, we have $a \in \pre(\mathbf{v}) - \pre(\mathbf{x}_2)$. 
Then, 
\begin{equation}\label{ain}
    a \in \pre(\mathbf{z}) - \pre(\mathbf{x}_2)
\end{equation}
because $\pre(\mathbf{v}) - \pre(\mathbf{x}_2) \subseteq \pre(\mathbf{z}) - \pre(\mathbf{x}_2)$. 
Let $b \in W$ be any element such that $\mathbf{x}_2^e R b R a$. 
Since $\mathbf{x}_2 \ast \seq{a} \sqsubseteq \mathbf{x} \sqsubseteq \pre(\mathbf{y})$, we have $b \notin \pre(\mathbf{y}) - \pre(\mathbf{x}_2)$. 
Then, we obtain $b \notin \pre(\mathbf{z}) - \pre(\mathbf{x}_2)$ because $\pre(\mathbf{y}) - \pre(\mathbf{x}_1) \supseteq \pre(\mathbf{z}) - \pre(\mathbf{x}_1)$ and $\mathbf{x}_1 = \mathbf{x}_2$. 
By combining this with (\ref{ain}), we get $\mathbf{x}_2 \ast \seq{a} \sqsubseteq \pre(\mathbf{z})$. 
Since $\mathbf{x}_2 \ast \seq{a} \sqsubseteq \pre(\mathbf{v})$, we have $\mathbf{x}_2 \ast \seq{a} \sqsubseteq \mathbf{x}_2$ by the maximality of $\mathbf{x}_2$, a contradiction. 
So, we have proved that $\mathbf{x}_2 = \mathbf{x}$.

($\dagger$): We obtain $\pre(\mathbf{y}) - \pre(\mathbf{x}) \supseteq \pre(\mathbf{v}) - \pre(\mathbf{x})$ because
\[
\pre(\mathbf{y}) - \pre(\mathbf{x}_1) \supseteq \pre(\mathbf{z}) - \pre(\mathbf{x}_1) = \pre(\mathbf{z}) - \pre(\mathbf{x}_2) \supseteq \pre(\mathbf{v}) - \pre(\mathbf{x}_2).
\] 

($\ddagger$): 
This follows from $\pre(\mathbf{y}) - \pre(\mathbf{x}) = \pre(\mathbf{y}) - \pre(\mathbf{x}_{1})$ and $\pre(\mathbf{y}) - \pre(\mathbf{x}) \supseteq \pre(\mathbf{z}) - \pre(\mathbf{x}_{2})$. 

We have proved that $S'$ is transitive. 
\end{itemize}

Finally, We prove the following claim. 

\begin{cl}\label{cl2}
For any modal formula $D$ and $\mathbf{x} \in W'$, 
\[
\mathbf{x} \Vdash' D \iff \mathbf{x}^e \Vdash D. 
\]
\end{cl}

\begin{proof}
We prove the claim by induction on the construction of $D$. 
We only prove the case $D \equiv B \rhd C$. 

\medskip

($\Rightarrow$): Assume $\mathbf{x} \Vdash' B \rhd C$. 
We show $\mathbf{x}^e \Vdash B \rhd C$. 
Suppose $\mathbf{x}^e R y$ and $y \Vdash B$. 
Let $\mathbf{y}:= \mathbf{x} * \seq{y}$, then obviously $\mathbf{y} \in W'$ and $\mathbf{x} R' \mathbf{y}$. 
By the induction hypothesis, $\mathbf{y} \Vdash' B$. 
By our assumption, there exists $\mathbf{z} \in W'$ such that $\mathbf{x} R' \mathbf{z}$, $\mathbf{y} S' \mathbf{z}$, and $\mathbf{z} \Vdash' C$. 
Since $\pre(\mathbf{y}) = \mathbf{x}$ and $\mathbf{x} \sqsubseteq \pre(\mathbf{z})$, we have $\mathbf{x} = \pre(\mathbf{y}) \cap \pre(\mathbf{z})$. 
Also, since $\pre(\mathbf{y}) - \pre(\mathbf{x}) = \{\mathbf{x}^{e}\}$, it follows from $\mathbf{y} S' \mathbf{z}$ that 
\[
(\exists a_{0},\ldots,a_{l+1} \in W)\bigl(a_{0} = \mathbf{y}^{e} \, \& \, a_{l+1} = \mathbf{z}^{e} \, \&  \, (\forall i \leq l)(a_{i} S_{\mathbf{x}^{e}} a_{i+1})\bigr). 
\]
By Proposition~\ref{FC}.2, $S_{\mathbf{x}^{e}}$ is transitive, and hence $\mathbf{y}^{e} S_{\mathbf{x}^{e}} \mathbf{z}^{e}$. 
By the induction hypothesis, $\mathbf{z}^e \Vdash C$, and therefore $\mathbf{x}^e \Vdash B \rhd C$. 

\medskip

($\Leftarrow$): Assume $\mathbf{x}^e \Vdash B \rhd C$. 
We show $\mathbf{x} \Vdash' B \rhd C$. 
Suppose $\mathbf{x} R' \mathbf{y}$ and $\mathbf{y} \Vdash' B$. 
Then, $\mathbf{x}^e R \mathbf{y}^e$ and $\mathbf{y}^e \Vdash B$ by the induction hypothesis. 
By our assumption, there exists $z \in W$ such that $\mathbf{y}^e S_{\mathbf{x}^e} z$ and $z \Vdash C$. 
Let $\mathbf{z} := \mathbf{x}*\seq{z}$. 
By Propositions~\ref{FC}.3, we have $\mathbf{x}^e R z$. 
Thus, $\mathbf{z} \in W'$ and $\mathbf{x} R' \mathbf{z}$. 
Since $\pre(\mathbf{z}) = \mathbf{x}$ and $\mathbf{x} \sqsubseteq \pre(\mathbf{y})$, we have $\mathbf{x} = \pre(\mathbf{y}) \cap \pre(\mathbf{z})$. 
Then, $\pre(\mathbf{z}) - \pre(\mathbf{x}) = \{\mathbf{x}^{e}\}$, and hence $\pre(\mathbf{y}) - \pre(\mathbf{x}) \supseteq \pre(\mathbf{z}) - \pre(\mathbf{x})$. 
Also, we have $\mathbf{x}^{e} \in \pre(\mathbf{y}) - \pre(\mathbf{x})$ and $\mathbf{y}^e S_{\mathbf{x}^e} \mathbf{z}^{e}$. 
Therefore, we obtain $\mathbf{y} S' \mathbf{z}$. 
By the induction hypothesis, $\mathbf{z} \Vdash' C$, and therefore we conclude $\mathbf{x} \Vdash' B \rhd C$. 
\end{proof}
By the claim, we conclude $\seq{w} \nVdash' A$. 
\end{proof}

\subsection{The modal completeness of $\IL^-(\J{2}_{+}, \J{5})$}

In this subsection, we investigate the logic $\IL^-(\J{2}_{+}, \J{5})$. 
The logic $\IL^-(\J{2}_{+}, \J{5})$ is complete with respect to the class of all $\IL^-(\J{2}_{+}, \J{5})$-Visser frames, however, it does not have the finite frame property. 

At first, we show that $\IL^-(\J{2}_{+}, \J{5})$ lacks the finite frame property. 
For this, we prepare the so-called generated submodel lemma. 
Let $M = (W, R, S, \Vdash)$ be any $\IL^-(\J{4}_+)$-Visser model. 
For each $r \in W$, let $M^\ast = (W^*, R^*, S^*, \Vdash^*)$ be the $\IL^-(\J{4}_+)$-Visser model simply obtained from $M$ by restricting $W$ to $R[r] \cup \{r\}$. 
We call the model $M^\ast$ the \textit{submodel of $M$ generated by $r$}.

\begin{lem}[The Generated Submodel Lemma]\label{gsl}
Let $M = (W, R, S, \Vdash)$ be any $\IL^-(\J{4}_+)$-Visser model, $r \in W$, and $M^* = (W^*, R^*, S^*, \Vdash^*)$ be the generated submodel of $M$ by $r$. 
Then, for any $x \in W^*$ and any modal formula $A$, 
\[
    x \Vdash A \iff x \Vdash^* A.
\]
\end{lem}
\begin{proof}
This lemme is proved by induction on the construction of $A$. 
In particular, the assumption `$x \Vdash B \rhd C$ and $x R y$' implies the existence of $z \in W$ such that $x R z$, $y S z$, and $z \Vdash C$. 
Since $x \in R[r] \cup \{r\}$, we have $z \in R[r] \subseteq W^*$. 
This observation shows the equivalence of the lemma in the case of $A \equiv B \rhd C$.  
\end{proof}

In \cite[Proposition 10]{IKO20}, the authors with Iwata proved that for any modal formula $A$, $\IL \vdash A$ if and only if 
\[
    \IL^-(\J{2}_{+}, \J{5}) \vdash \bigwedge\{(C \rhd C) \land \Box(C \rhd C) \mid C \in \mathrm{Sub}(A)\} \to A, 
\]
where $\Sub(A)$ is the set of all subformulas of $A$. 
By using this equivalence, we then proved that the Craig interpolation property of $\IL$ follows from that of $\IL^-(\J{2}_{+}, \J{5})$. 
We prove the following proposition based on the same idea.

\begin{prop}\label{Cex}
The logic $\IL^-(\J{2}_{+}, \J{5})$ does not have finite frame property with respect to $\IL^-(\J{2}_{+}, \J{5})$-Visser frames. 
\end{prop}

\begin{proof}
Suppose that $\IL^-(\J{2}_{+}, \J{5})$ has the finite frame property with respect to $\IL^-(\J{2}_{+}, \J{5})$-Visser frames. 
Then, we show that $\IL$ also has the finite frame property with respect to $\IL$-Visser frames. 
This contradicts Fact~\ref{Fact:SComp}.2. 

Suppose $\IL \nvdash A$. 
It suffices to show that there exists a finite $\IL$-Visser frame in which $A$ is not valid. 
Since $\IL$ proves $C \rhd C$ for any modal formula $C$, we have 
\[
\IL^-(\J{2}_{+}, \J{5}) \nvdash \bigwedge\{(C \rhd C) \land \Box(C \rhd C) \mid C \in \mathrm{Sub}(A)\} \to A. 
\] 
By our supposition, there exist a finite $\IL^-(\J{2}_{+}, \J{5})$-Visser frame $\mathcal{F} = (W, R, S)$ and $w \in W$ such that $w \Vdash \bigwedge\{(C \rhd C) \land \Box(C \rhd C) \mid C \in \mathrm{Sub}(A)\}$ and $w \nVdash A$. 
By the generated submodel lemma, we may assume that $w$ is the root of $W$, that is, $W = R[w] \cup \{w\}$. 
Hence, for any $x \in W$ and any $C \in \mathrm{Sub}(A)$, we have $x \Vdash C \rhd C$. 
We define the binary relation $S'$ and the satisfaction relation $\Vdash'$ as follows:
\begin{itemize}
	\item $y S' z :\iff y S z$ or $y = z$.
	\item $x \Vdash' p :\iff x \Vdash p$. 
\end{itemize}
Obviously, $\mathcal{F}' = (W, R, S')$ is a finite $\IL$-Visser frame. 
We then prove the following claim. 

\begin{cl}\label{cl6}
For any modal formula $D \in \Sub(A)$ and $x \in W$, 
\[
x \Vdash' D \iff x \Vdash D. 
\]
\end{cl}

\begin{proof}
We prove the claim by induction on the construction of $D$. 
We only prove the case $D \equiv B \rhd C$. 

\medskip

($\Rightarrow$): Assume $x \Vdash' B \rhd C$ and we show $x \Vdash B \rhd C$. 
Suppose $x R y$ and $y \Vdash B$. 
By the induction hypothesis, $y \Vdash' B$, and hence there exists $z \in W$ such that $x R z$, $y S' z$, and $z \Vdash' C$.  
By the induction hypothesis, $z \Vdash C$.  
If $y S z$, then we are done. 
If $y = z$, then $y \Vdash C$, and so $x \Vdash C \rhd C$ implies that there exists $z' \in W$ such that $x R z'$, $y S z'$, and $z' \Vdash C$. 
In either case, we obtain $x \Vdash B \rhd C$. 

\medskip

($\Leftarrow$): This is obviously proved. 
\end{proof}

Since $w \nVdash A$, we obtain $w \nVdash' A$ by the claim. 
Therefore, we conclude $A$ is not valid in $\mathcal{F}$.
\end{proof}

Next, we prove the modal completeness $\IL^-(\J{2}_{+}, \J{5})$. 
Our proof is based on Ignatiev's construction of models presented in~\cite{Ign91}. 

\begin{thm}\label{SVIL}
For any modal formula $A$, the following are equivalent: 
\begin{enumerate}
	\item $\IL^-(\J{2}_{+}, \J{5}) \vdash A$. 
	\item $A$ is valid in all $\IL^-(\J{2}_{+}, \J{5})$-Visser frames.
\end{enumerate}
\end{thm}

\begin{proof}
We only prove the implication $(2 \Rightarrow 1)$. 
Suppose $\IL^-(\J{2}_{+}, \J{5}) \nvdash A$. 
By Fact~\ref{SoIC}, there exist a finite $\IL^-\!$-model $(W, R, \{S_{x}\}_{x \in W}, \Vdash)$ and $w \in W$ such that $w \nVdash A$ and both $\J{2}_{+}$ and $\J{5}$ are valid in the frame $(W, R, \{S_{x}\}_{x \in W})$. 
We may assume $0 \notin W$. 
For each finite sequence $\Gamma$ of elements of $W \cup \{0\}$, let $|\Gamma|$ denote the length of $\Gamma$. 
We define the set $W'$ as follows:  
\begin{eqnarray*}
W' := \{(\Gamma, \Delta) \mid &1)& \Gamma \, \text{is a finite sequence of elements of} \ W \, \text{and} \, \Gamma \neq \varepsilon,\\
&2)& \Delta \, \text{is a finite sequence of elements of} \, W \cup \{0\}  \, \text{and}\\
&& \Delta=\varepsilon \, \text{is allowed,} \\
&3)& |\Delta| + 1 = |\Gamma|, \\
&4)& \text{for} \, \Gamma =\seq{x_{0},\ldots,x_{n}} \, \text{and} \, \Delta = \seq{v_{0},\ldots,v_{n-1}}, \\
&& (\forall i < n) [(v_{i} = 0 \Rightarrow x_{i} R x_{i+1}) \land (v_{i} \neq 0 \Rightarrow x_{i} S_{v_{i}} x_{i+1})]\}.
\end{eqnarray*}

In the following, $\mathbf{x}, \mathbf{y}, \mathbf{z},\ldots$ denote elements of $W'$. 
Here, we prepare several notation. 
In this paragraph, let $\mathbf{x}$ and $\mathbf{y}$ be $(\seq{x_{0},\ldots,x_{n}}, \seq{v_{0},\ldots,v_{n-1}})$ and $(\seq{y_{0},\ldots,y_{m}}, \seq{w_{1},\ldots,w_{m-1}})$, respectively.
\begin{itemize}
	\item $\mathbf{x} \sqsubset \mathbf{y} : \iff$ the first and the second components of $\mathbf{x}$ are proper initial segments of those of $\mathbf{y}$, respectively. 
    In other words, $n < m$, $(\forall i \leq n)(x_{i} = y_{i})$, and $(\forall j < n)(v_{j} = w_{j})$. 
    \item $\mathbf{x} \sqsubseteq \mathbf{y} : \iff \mathbf{x} \sqsubset \mathbf{y}$ or $\mathbf{x} = \mathbf{y}$. 
    \item $k_\mathbf{x} : = \min \{j \mid \forall i\, (j \leq i < n \Rightarrow v_{i-1} \neq 0)\}$. 
	\item $K(\mathbf{x}) := (\seq{x_{0},\ldots,x_{k_\mathbf{x}}}, \seq{v_{0},\ldots,v_{k_\mathbf{x}-1}})$. 
	\item $|\mathbf{x}| := n-1$. 
	\item $\mathbf{x}^{e} := x_{n}$. 
\end{itemize}
It follows from the definitions that for any $\mathbf{x} \in W'$, we have $k_{\mathbf{x}} \leq |\mathbf{x}|$, $K(\mathbf{x}) \sqsubseteq \mathbf{x}$, and $K(\mathbf{x}) \in W'$.  
Also, $\mathbf{x} \sqsubseteq \mathbf{y}$ implies $k_\mathbf{x} \leq k_\mathbf{y}$, and hence $K(\mathbf{x}) \sqsubseteq K(\mathbf{y})$. 
We define the binary relations $R'$ and $S'$ on $W'$ and the satisfaction relation $\Vdash'$ as follows:
\begin{itemize}
	\item $\mathbf{x} R' \mathbf{y} :\iff$ $\mathbf{x} \sqsubset \mathbf{y}$ and 
\begin{equation}\label{R'cond}
\forall i \, \bigl(n \leq i < m \, \& \, v_{i} \neq 0 \Rightarrow \exists u\,  (n \leq u \leq i \, \& \, v_{i} = x_{u})\bigr).
\end{equation}

	\item $\mathbf{y} S' \mathbf{z} :\iff K(\mathbf{y}) \sqsubseteq K(\mathbf{z})$ and $|\mathbf{y}| < |\mathbf{z}|$. 

	\item $\mathbf{x} \Vdash' p :\iff \mathbf{x}^e \Vdash p$. 
\end{itemize}

We show that $\mathcal{F} := (W', R', S')$ is an $\IL^-(\J{2}_{+}, \J{5})$-Visser frame. 
Since $S'$ is obviously transitive and satisfies $R' \subseteq S'$, it suffices to prove that $R'$ is transitive and conversely well-founded. 

Firstly, we show that $R'$ is transitive. 
Suppose $\mathbf{x} R' \mathbf{y}$ and $\mathbf{y} R' \mathbf{z}$. 
We show $\mathbf{x} R' \mathbf{z}$. 
Let
\begin{itemize}
    \item $\mathbf{x} = (\seq{x_{0},\ldots,x_{n}}, \seq{v_{0},\ldots,v_{n-1}})$,
    \item $\mathbf{y} = (\seq{x_{0},\ldots,x_{m}}, \seq{v_{0},\ldots,v_{m-1}})$,
    \item $\mathbf{z} = (\seq{x_{0},\ldots,x_{l}}, \seq{v_{0},\ldots,v_{l-1}})$, where $n < m < l$. 
\end{itemize}

Let $i$ be such that $n \leq i < l$ and $v_i \neq 0$. 
\begin{itemize}
    \item If $n \leq i < m$, then by (\ref{R'cond}) of $\mathbf{x} R' \mathbf{y}$, there exists $u$ such that $n \leq u \leq i$ and $v_i = x_u$.
    \item If $m \leq i < l$, then by (\ref{R'cond}) of $\mathbf{y} R' \mathbf{z}$, there exists $u$ such that $m \leq u \leq i$ and $v_i = x_u$.
\end{itemize}
Hence, (\ref{R'cond}) of $\mathbf{x} R' \mathbf{z}$ is verified. 
Thus, we conclude that $\mathbf{x} R' \mathbf{z}$ holds. 

Secondly, we show that $R'$ is conversely well-founded. 
We prove the following claim. 

\begin{cl}\label{cl3}
For $\mathbf{x}, \mathbf{y} \in W'$, suppose $\mathbf{x} R' \mathbf{y}$, $\mathbf{x} = (\seq{x_{0},\ldots,x_{n}}, \seq{v_{0},\ldots,v_{n-1}})$, and $\mathbf{y} = (\seq{x_{0},\ldots,x_{m}}, \seq{v_{0},\ldots,v_{m-1}})$, where $n < m$. 
Then, for any $i$ with $n+1 \leq i \leq m$, we have $x_{n} R x_{i}$. 
\end{cl}

\begin{proof}
We prove the claim by induction on $i$. 
Let $i$ be such that $n+1 \leq i \leq m$ and suppose that the statement holds for all $j$ with $n+1 \leq j < i$. 
We distinguish the following two cases. 
\begin{itemize}
	\item[Case 1:] $v_{i-1} = 0$. 
 Then, $x_{i-1} R x_{i}$. 
If $i-1 = n$, then we are done. 
If $i-1 \geq n+1$, then by the induction hypothesis, $x_{n} R x_{i-1}$, and so we obtain $x_n R x_i$. 

	\item[Case 2:] $v_{i-1} \neq 0$. 
 Then, $x_{i-1} S_{v_{i-1}} x_{i}$. 
Since $n \leq i - 1 < m$, by (\ref{R'cond}), there exists $u$ such that $n \leq u \leq i-1$ and $v_{i-1} = x_{u}$. 
So, $x_{i-1} S_{x_u} x_{i}$, and hence we have $x_{u} R x_{i}$ by Fact~\ref{FC}.2. 
If $u = n$, then we are done. 
If $u \neq n$, then $n+1 \leq u < i$, and so by the induction hypothesis, $x_n R x_u$. 
Thus, we conclude $x_{n} R x_{i}$. \qedhere
\end{itemize}
\end{proof}

If $\mathbf{x} R' \mathbf{y}$, then by the claim, we have $\mathbf{x}^{e} R \mathbf{y}^{e}$. 
So, the converse well-foundedness of $R'$ follows from that of $R$.

Similar to the above claim, the following claim on the family $\{S_x\}_{x \in W}$ holds. 
\begin{cl}\label{cl4}
For $\mathbf{x}, \mathbf{y} \in W'$, suppose $\mathbf{x} R' \mathbf{y}$, $\mathbf{x} = (\seq{x_{0},\ldots,x_{n}}, \seq{v_{0},\ldots,v_{n-1}})$, and $\mathbf{y} = (\seq{x_{0},\ldots,x_{m}}, \seq{v_{0},\ldots,v_{m-1}})$, where $n < m$. 
Then, for any $i$ with $n+2 \leq i \leq m$, we have $x_{n+1} S_{x_{n}} x_{i}$. 
\end{cl}

\begin{proof}
We prove the claim by induction on $i$. 
Let $i$ be such that $n+2 \leq i \leq m$ and suppose that the statement holds for all $j$ with $n+2 \leq j < i$. 
We distinguish the following two cases. 
\begin{itemize}
	\item[Case 1:] $v_{i-1} = 0$. 
Then, $x_{i-1} R x_{i}$.
Since $n+1 \leq i-1 \leq m$, by the claim above, we have $x_{n} R x_{i-1}$. 
We then obtain $x_{i-1} S_{x_{n}} x_{i}$ by Fact~\ref{FC}.4. 
If $i-1= n+1$, then we are done. 
If $i-1\geq n+2$, then by the induction hypothesis, we get $x_{n+1} S_{x_{n}} x_{i-1}$. 
Then, we obtain $x_{n+1} S_{x_{n}} x_{i}$ by Fact~\ref{FC}.2.

	\item[Case 2:] $v_{i-1} \neq 0$. 
 Then, $x_{i-1} S_{v_{i-1}} x_{i}$. 
Since $n \leq i - 1 < m$, by (\ref{R'cond}), there exists a natural number $u$ such that $n \leq u \leq i-1$ and $v_{i-1} = x_{u}$, and hence $x_{i-1} S_{x_u} x_{i}$. 
If $u=n$, we then obtain $x_{n+1} S_{x_{n}} x_{i}$ as in Case 1. 
If $u \geq n + 1$, then we have $x_{n} R x_{u}$ by the claim above.  
Since $x_u R x_i$ follows from $x_{i-1} S_{x_u} x_{i}$ by Fact~\ref{FC}.3, we obtain $x_{u} S_{x_{n}} x_{i}$ by Fact~\ref{FC}.4. 
If $u = n+1$, then we are done. 
If $u \geq n+2$, then by the induction hypothesis, we have $x_{n+1} S_{x_{n}} x_{u}$. 
Then, we conclude $x_{n+1} S_{x_{n}} x_{i}$ by Fact~\ref{FC}.2. \qedhere
\end{itemize}
\end{proof}

Finally, we prove the following claim. 

\begin{cl}\label{cl5}
For any modal formula $D$ and $\mathbf{x} \in W'$, 
\[
\mathbf{x} \Vdash' D \iff \mathbf{x}^e \Vdash D. 
\]
\end{cl}

\begin{proof}
Let $\mathbf{x} = (\seq{x_{0},\ldots,x_{n}}, \seq{v_{0},\ldots,v_{n-1}})$. 
We prove the claim by induction on the construction of $D$. 
We only prove the case $D \equiv B \rhd C$. 

\medskip

($\Rightarrow$): Assume $\mathbf{x} \Vdash' B \rhd C$. 
We show $\mathbf{x}^e \Vdash B \rhd C$. 
Suppose $\mathbf{x}^e R y$ and $y \Vdash B$. 
Let $\mathbf{y}:=  (\seq{x_{0},\ldots,x_{n}, y}, \seq{v_{0},\ldots,v_{n-1}, 0})$. 
Then, we have $\mathbf{y} \in W'$ and $\mathbf{x} \sqsubset \mathbf{y}$. 
Also, since the condition (\ref{R'cond}) of $\mathbf{x} R' \mathbf{y}$ trivially meets, we have that $\mathbf{x} R' \mathbf{y}$ holds. 
By the induction hypothesis, we have $\mathbf{y} \Vdash' B$, and so there exists $\mathbf{z} \in W'$ such that $\mathbf{x} R' \mathbf{z}$, $\mathbf{y} S' \mathbf{z}$, and $\mathbf{z} \Vdash' C$. 
Since $\mathbf{x} \sqsubset \mathbf{z}$, we have that $\mathbf{z}$ is of the form
\[
    (\seq{x_{0},\ldots,x_{n}, x_{n+1}, \ldots, x_{m}}, \seq{v_{0},\ldots, v_{n-1}, v_{n}, \ldots, v_{m-1}})
\]
for some $m > n$. 
By the definition of $\mathbf{y}$, we have $k_{\mathbf{y}} = n$, and so $\mathbf{y} = K(\mathbf{y})$. 
Also, by the definition of $S'$ we have $K(\mathbf{y}) \sqsubseteq K(\mathbf{z})$. 
Thus $\mathbf{y} \sqsubseteq K(\mathbf{z})$. 
This means that $x_{n+1} = y$. 
Moreover, since $|\mathbf{y}| < |\mathbf{z}|$, we have $m \geq n+2$. 
By the second claim, we obtain $x_{n+1} S_{x_n} x_m$, that is, $y S_{\mathbf{x}^e} \mathbf{z}^e$. 
By the induction hypothesis, $\mathbf{z}^e \Vdash C$, and hence $\mathbf{x}^e \Vdash B \rhd C$. 

\medskip

($\Leftarrow$): Assume $\mathbf{x}^e \Vdash B \rhd C$. 
We show $\mathbf{x} \Vdash' B \rhd C$. 
Suppose $\mathbf{x} R' \mathbf{y}$ and $\mathbf{y} \Vdash' B$. 
Then, $\mathbf{y}$ is of the form $(\seq{x_{0},\ldots,x_{n}, x_{n+1}, \ldots, x_{m}}, \seq{v_{0},\ldots, v_{n-1}, v_{n}, \ldots, v_{m-1}})$ for some $m > n$. 
By the first claim, we have $x_n R x_m$, that is, $\mathbf{x}^e R \mathbf{y}^e$. 
Since $\mathbf{y}^e \Vdash B$ by the induction hypothesis, there exists $z \in W$ such that $\mathbf{y}^e S_{\mathbf{x}^e} z$ and $z \Vdash C$. 
Let $\mathbf{z} := (\seq{x_{0},\ldots,x_{m}, z}, \seq{v_{0},\ldots, v_{m-1}, \mathbf{x}^e})$. 
Obviously, we have $\mathbf{z} \in W'$ and $\mathbf{x} R' \mathbf{z}$. 
Then, $|\mathbf{y}| < |\mathbf{z}|$. 
Also, since $k_\mathbf{y} = k_\mathbf{z}$, we have $K(\mathbf{y}) = K(\mathbf{z})$.  
Thus, $\mathbf{y} S' \mathbf{z}$. 
By the induction hypothesis, $\mathbf{z} \Vdash' C$, and hence $\mathbf{x} \Vdash' B \rhd C$. 
\end{proof}
By the claim, we conclude $(\seq{w}, \varepsilon) \nVdash' A$. 
\end{proof}

\section{Concluding remarks}

In this paper, we investigated the modal completeness and the finite frame property of several extensions of $\IL^-(\J{4}_+)$ with respect to Visser frames. 
Among other things, we proved that $\CL$ has the finite frame property with respect to $\CL$-Visser frames (Theorem~\ref{SV2}). 
Then, Ignatiev's proof of the arithmetical completeness of $\CL$ (Theorem \ref{ACT}) is now verified.

Moreover, our result is applicable to the completeness of these logics with respect to topological semantics. 
In~\cite{IK2021}, a topological semantics of extensions of $\CL$ was introduced. 
Also, it is proved that $\CL$ is sound and complete with respect to that topological semantics by translating every $\CL$-Visser frame into a bitopological space. 
Since Ignatiev's counter $\CL$-Visser model is infinite, it is asked in \cite[Problem 6.3]{IK2021} whether $\CL$ has the finite frame property with respect to topological semantics or not. 
Theorem~\ref{SV2} of the present paper also answers this problem affirmatively.

\begin{cor}
The logic $\CL$ has the finite frame property with respect to Iwata and Kurahashi's topological semantics. 
\end{cor}

\section*{Acknowledgement}

The first author was supported by Foundation of Research Fellows, The Mathematical Society of Japan. 
The second author was supported by JSPS KAKENHI Grant Numbers JP19K14586 and JP23K03200.



\bibliographystyle{plain}
\bibliography{ref}

\end{document}